\documentclass[11pt,reqno]{article}
\usepackage{amssymb}
\usepackage{amsmath}
\usepackage{amsthm}
\usepackage{amsfonts}
\usepackage{url}
\usepackage{hyperref}
\usepackage{breakurl}
\usepackage{comment}

\renewcommand{\ge}{\geqslant}
\renewcommand{\le}{\leqslant}

\newcommand{\C}{{\mathbb C}}
\newcommand{\R}{{\mathbb R}}

\newcommand{\Z}{{\mathbb Z}}
\newcommand{\primes}{{\mathcal P}}

\theoremstyle{plain}

\theoremstyle{definition}

\theoremstyle{claim}
\newtheorem{claim}{Claim}
\newtheorem{remark}{Remark}

\begin{document}
\bibliographystyle{plain}
\title{Some instructive mathematical errors%
}
\author
{Richard P.\ Brent\footnote{Mathematical Sciences Institute, 
Australian National University, Canberra, Australia}
}
\date{\today}	
\maketitle

\begin{abstract}

We describe various errors in the mathematical literature, and consider how
some of them might have been avoided, or at least detected at an earlier stage,
using tools such as Maple or Sage.
Our examples are drawn from three broad categories of errors. 
First, we consider some significant errors made by highly-regarded
mathematicians. In some cases these errors were not detected until many
years after their publication.
Second, we consider in some detail an error 
that was recently detected by the author.
This error in a refereed journal led to further errors by
at least one author who relied on the (incorrect) result.
Finally, we mention some instructive
errors that have been detected in the author's own
published papers.

\end{abstract}

\section{Introduction}				\label{sec:Intro}

\begin{quotation}
\emph{Those that fail to learn from history are doomed to repeat it.}\\[5pt]
\hspace*{\fill}Winston Churchill~\cite{Churchill}.
\end{quotation}

Since mathematics is a human endeavour, errors can and do occur.
It is worth studying past errors and learning from them to reduce
the likelihood of making similar errors in the future.

We are concerned with nontrivial errors made by working mathematicians, not
with the class of errors that school-children often make in their homework.

\vspace*{\fill}
\pagebreak[3]

The errors that we consider can be grouped into three broad categories.

\begin{enumerate}
\item Well-known errors made by prominent mathematicians
	(see \S\ref{sec:well-known}).
						\label{well-known-errors}
\item Errors discovered by the author in other mathematicians' work 
	(\S\ref{sec:VMA}).
						\label{recent-errors}
\item Some errors in, or relevant to, the author's own work
	(\S\ref{sec:Brent}).			\label{rpb-errors}
\end{enumerate}

It is worth mentioning the errors in category \ref{well-known-errors}
in order to show that
even the best mathematicians are fallible.
Also, correcting an error may lead to interesting mathematics
(see for example~\S\ref{subsec:Poincare} and \S\ref{subsec:Wiles}).
Our list is by no means exhaustive.

Errors in category $\ref{recent-errors}$ were discovered recently by the
author. These errors are discussed in some detail in order to show how
they were detected, and how they might have been avoided.

Errors in category $\ref{rpb-errors}$ are described as a penance, and
because the reader may find them instructive.

Our definition of ``error'' is rather broad. 
An error in a proof might be an unjustified assumption, or
a gap that needs to be filled.
As well as errors in published
or submitted papers, we include some incorrect results that were announced
in other ways, some implicit errors 
(see \S\ref{subsec:rpb269}),
and some claims that were expressed as conjectures rather than theorems
(see \S\ref{subsec:Stieltjes}, \S\ref{subsec:rpb055}).

\section{Some well-known errors}		       \label{sec:well-known}

Many errors have been made by highly-regarded
mathematicians. They have generally
been well-documented, 
so we merely give an overview, with references to further information.
Our list is (approximately) in chronological order.
A considerably longer list is available online~\cite{Wiki-incomplete-proofs}.

\subsection{The four-colour theorem}		\label{subsec:four-colour}

Many fallacious proofs of the four-colour theorem have been given.
Notable are two, by Kempe~\cite{Kempe} (1879) and Tait~\cite{Tait} (1880),
that both stood for
$11$ years before the errors were noticed~\cite{Wiki-four-colour}.
The first correct proof, by Appel and Haken~\cite{Appel-Haken}
in 1976, used some
of Kempe's ideas. The Appel-Haken proof was at first controversial, because
it depended on the computer-aided verification of $1834$ ``reducible
configurations'' which could not feasibly be verified by hand.
In fact, some relatively minor errors were discovered later, and corrected
in the 1989 book by Appel and Haken~\cite{Appel-Haken-book}.
There is still no known ``simple'' proof that does not involve checking
a large number of cases.

\subsection{Mertens and Stieltjes} 		\label{subsec:Stieltjes}

In 1897, Franz Mertens~\cite{Mertens} conjectured, on the basis
of numerical evidence, that 
$|M(x)| \le x^{1/2}$ for all $x > 1$,
where $M(x) := \sum_{n\le x}\mu(n),$
and $\mu(n)$ is the M\"obius function.
The same inequality was 
also conjectured (in 1885) by Thomas Stieltjes, 
in a letter to Charles Hermite~\cite{Stieltjes,Wiki-Mertens}.

In the same letter to Hermite, Stieltjes claimed to have a
proof of the weaker result
that $M(x) = O(x^{1/2})$. However, he never published 
a proof, and none was found in his papers after his death. 

It is well-known~\cite[\S14.28]{Titchmarsh}
that $M(x) = O(x^{1/2 + \varepsilon})$ for 
all $\varepsilon > 0$ if and only if the Riemann Hypothesis (RH)
is true. Thus, the claimed result of Stieltjes would imply RH.
In fact, it would imply even more.

In 1985, Andrew Odlyzko and Herman te Riele~\cite{Odlyzko-teRiele}
disproved the conjecture of Mertens.
The disproof was numerical, and did not disprove the claim
of Stieltjes. However, for reasons given 
in~\cite[\S\S1--2]{Odlyzko-teRiele}, most experts
believe that the claim $M(x)=O(x^{1/2})$
of Stieltjes is false.

\subsection{Poincar\'e's prize essay}		\label{subsec:Poincare}

In 1888, Henri Poincar\'e submitted a paper, entitled
\emph{The Three-Body Problem and the Equations of Dynamics},
to a competition sponsored by King Oscar~II of Sweden and Norway, and
the journal \emph{Acta Mathematica} (edited by G\"osta Mittag-Leffler). 
The prize committee 
(Hermite, Mittag-Leffler, and Weierstrass) awarded Poincar\'e the prize. 
The committee stated:
\begin{quotation}
\emph{It is the deep and original work of a 
mathematical genius whose position is
among the greatest mathematicians of the century.  The most important and
difficult questions, like the stability of the world system, are treated
using methods which open a new era in celestial mechanics.}
\end{quotation}

While Poincar\'e's manuscript was being prepared for printing in
\emph{Acta Mathematica}, sub-editor Edvard Phragm\'en communicated with
Poincar\'e about parts of the manuscript that he found difficult to
understand. 

\begin{quotation}
\emph{If the author were not what he is, I would not for a\\
moment hesitate to say that he has made a great mistake here.}\\[5pt]
\hspace*{\fill}
Phragm\'en, December 1888, see~\cite{Phragmen-StAndrews}.
\end{quotation}
Eventually, in December 1889, Poincar\'e admitted that he had made an error
with a critical consequence~-- his claimed proof of the stability of the
solar system was invalid!  Unfortunately, the paper~\cite{Poincare-1889b}
had already been printed,
and Mittag-Leffler had started to distribute it. With some difficulty,
Mittag-Leffler managed to recall almost all of the copies.
For further details and references, see~\cite{Gray-Poincare}.

Poincar\'e prepared a corrected version, about twice as long as the original
prize entry, and it was eventually published~\cite{Poincare-1890c}. 
Poincar\'e had to pay the extra costs involved, which exceeded the prize
money that he had won.

The story had a happy ending as, in realising his error and making
his corrections, Poincar\'e discovered the phenomenon of 
chaos~\cite{Barrow-Green,Diacu,Gray,Mawhin,Poincare-StAndrews}.
Some of the credit for this great discovery must go to Phragm\'en.

\subsection{Rademacher and the Riemann Hypothesis}
						\label{subsec:Rademacher}

\begin{quotation}
\emph{If you want to climb the Matterhorn you might first wish to go to 
Zermatt where those who have tried are buried.}\\
\hspace*{\fill}George P\'olya\footnote{Quoted by 
Lars H\"ormander, see~\cite[pg.~69]{BCRW}.}
\end{quotation}

In 1943, Hans Rademacher~\cite{Berndt-Rademacher} 
submitted a proof\footnote{Time~\cite{Time} says
``disproving the Riemann Hypothesis'', but this is likely to be an
error, see~\cite{Mathoverflow-Rademacher}.}
of the Riemann Hypothesis (RH)
to \emph{Transactions of the American Mathematical Society}. It was
accepted and was scheduled to appear in the May 1945 issue.
At the last moment, Rademacher withdrew the paper, because
Carl Siegel had found a flaw in his reasoning.
The paper did not appear in the \emph{Transactions}, but the story was
told to a wide audience in \emph{Time} magazine~\cite{Time}.

Since Rademacher's manuscript is not available, we can not say what his
error was. However, there have been
suggestions (see e.g.~\cite{Mathoverflow-Rademacher}) that he used a proof by
contradiction. In other words, he assumed the existence of a zero
$\rho$ of $\zeta(s)$ with $\Re(\rho) > 1/2$, and from this assumption he
derived a contradiction, implying that $\Re(\rho) \le 1/2$. This would prove
RH. However, in any proof by contradiction, one has to be extremely careful,
as an error in the proof might lead to an erroneous contradiction.

We remark that several other mathematicians have claimed to prove RH. 
Some serious attempts are mentioned in~\cite[Ch.~8]{BCRW}.
For other attempts, see~\cite{Watkins}.
Roger Heath-Brown (quoted in~\cite[pg.~112]{Sabbagh}) comments:
\begin{quotation}
\emph{I receive unsolicited manuscripts quite frequently~-- one finds a
particular person who has an idea, and no matter how many times you point
out a mistake, they correct that and produce something else that is also a
mistake.}
\end{quotation}

\subsection{Wiles and the proof of Fermat's Last Theorem}
						\label{subsec:Wiles}

What we now call \emph{Fermat's Last Theorem} (FLT) was mentioned by
Pierre de Fermat around 1637 in the margin of his copy of Diophantus's
\emph{Arithmetica}. The story is too well-known to warrant repeating here.

There have been a great many erroneous proofs of FLT, 
see~\cite{Wiki-FLT-history}.
The list might well include Fermat's proof, referred to in his 
marginal note~-- we will never know with certainty,
but in view of developments over the next
$350$ years, it seems very unlikely that his proof was correct.

In June 1993, Andrew Wiles announced a proof of the Taniyama-Shimura conjecture for
semistable elliptic curves. By previous work of Ribet and of Frey (building on
earlier work by other mathematicians), this was known to imply FLT.

Wiles presented his proof at a series of lectures (in Cambridge, UK),
and submitted a paper to \emph{Annals of Mathematics}. 
However, while refereeing the
paper, Nick Katz found a gap in the proof. 

Wiles worked to repair his proof, first alone, and then with his former
student Richard Taylor.  By September 1994 they were almost ready to admit
defeat.  Then, while trying to understand why his approach could not be made to
work, Wiles had a sudden insight.
\begin{quotation}
\emph{I was sitting at my desk examining the Kolyvagin–Flach method.  It wasn't
that I believed I could make it work, but I thought that at least I could
explain why it didn't work.  Suddenly I had this incredible revelation.  I
realised that, the Kolyvagin–Flach method wasn't working, but it was all I
needed to make my original Iwasawa theory work from three years earlier.  So
out of the ashes of Kolyvagin–Flach seemed to rise the true answer to the
problem.  It was so indescribably beautiful; it was so simple and so
elegant.}\\[5pt]
\hspace*{\fill}Andrew Wiles, quoted by Simon Singh~\cite{Singh}.
\end{quotation}

Wiles's insight led to a corrected proof of FLT, published in 1995 as two
papers, one by Wiles~\cite{Wiles}, and a companion paper by Taylor and 
Wiles~\cite{Taylor-Wiles}.
For a non-technical overview, see \cite{Singh,Wiki-FLT-proof}.

\begin{table}[ht]
\begin{center}
\begin{tabular}{l l l}\hline\\[-5pt]
Author			& \textbf{Poincar\'e} 		& \textbf{Wiles}\\[2pt]
Topic			& \textbf{Three-body problem}   & \textbf{FLT}\\[2pt]
Date			& $1888$--$1892$		&$1993$--$1995$\\[2pt]
Error(s) found by	& Phragm\'en		 	& Katz \\[2pt]
Corrected by		& Poincar\'e			& Wiles and 
							  Taylor\\[2pt]
Some further		& Theory of chaos in		& Modularity theorem\\
developments		& dynamical systems		& proved by Breuil,\\
			& developed by Poincar\'e	& Conrad, Diamond,\\
			& and later authors		& and Taylor 
							  (2001)\\[5pt]
\hline
\end{tabular}
\end{center}
\vspace*{-10pt}
\caption{Comparing the  sagas of Poincar\'e and Wiles}
\label{tab:comparison}
\end{table}

The story (or perhaps we can justifiably call it a \emph{saga}) 
of Wiles's proof of FLT has some
striking analogies to the story/saga of Poincar\'e's prize \hbox{essay}
(\S\ref{subsec:Poincare}). Wiles's original manuscript had an error (discovered
by Katz); Poincar\'e's had errors (discovered by Phragm\'en).
Both papers contained other results that were of value, but the errors
invalidated their main results. The errors were corrected, though not without
difficulty,
and the corrected results were published and have been accepted as valid. 
Both papers led to further
developments: in the case of Wiles,
proof of the Taniyama-Shimura-Weil conjecture for all elliptic
curves (a result now known as the \emph{modularity theorem})
by Breuil, Conrad, Diamond, and Taylor~\cite{modularity};
in the case of  Poincar\'e,  
the development of chaos theory by Poincar\'e and many subsequent authors.
The analogies are summarised in Table~\ref{tab:comparison}.

\subsection{Mochizuki's claimed proof of the \emph{abc} conjecture}
						\label{subsec:Mochizuki}

The \emph{abc conjecture}\footnote{Also, perhaps more informatively, known
as the \emph{Oesterl\'e-Masser conjecture} after its proposers
David Masser~\cite{Masser} and Joseph Oesterl\'e~\cite{Oesterle}.}
considers three relatively prime positive integers $a, b, c$
such that $a+b=c$. Let $d$ be the product of the distinct prime factors of
$abc$ (hence the name ``$abc$ conjecture''). 
The conjecture gives a lower bound on $d$ that (conjecturely) holds
for all but a finite number of cases.\footnote{The bound on the finite number
$N(\varepsilon)$ of exceptional cases depends on a positive but arbitrarily 
small parameter $\varepsilon$; for details see~\cite{Wiki-abc}.}

The $abc$ conjecture was shown by Oesterl\'e~\cite{Oesterle} to be
essentially equivalent to the Szpiro conjecture,
which relates the conductor and the discriminant of an elliptic curve.
If true, these (equivalent) conjectures 
would have many interesting consequences.
Some of these consequences are already
known to be true, e.g.\ Roth's theorem and
the theorem of Faltings
(formerly known as the Mordell conjecture);
others still have the status of conjectures, 
e.g.\ the Fermat-Catalan conjecture, which is almost\footnote{Not quite
a generalisation, because it allows a finite number of exceptions.}
a generalisation of FLT.

In August 2012, Shinichi Mochizuki claimed a proof of Szpiro's conjecture 
(and hence the $abc$ conjecture) by
developing a new theory called 
\emph{inter-universal Teichmüller theory} (IUTT)~\cite{Ball}.
However, the proof has not been generally 
accepted by the mathematical community~\cite{Wiki-Szpiro}.
On one side we have views such as:
\begin{quotation}
\emph{We, 
the authors of this note, came to the conclusion that there is no proof. 
We are going to explain where, in our opinion, the suggested proof has a
problem, a problem so severe that in our opinion small modifications will
not rescue the proof strategy.  We supplement our report by mentioning
dissenting views from Prof.\ Mochizuki and Prof.\ Hoshi about the issues we
raise with the proof and whether it constitutes a gap at all.}\\
\hspace*{\fill}Peter Scholze and Jakob Stix~\cite{Scholze-Stix-2018}.
\end{quotation}
On the other side, Mochizuki still claims that his proof is correct,
and has published it in a refereed journal~\cite{Castelvecchi}.
Mochizuki was at that time
the editor-in-chief of the journal.\footnote{As many of the examples 
in~\cite{Wiki-incomplete-proofs} show, 
publication in a refereed journal is not in itself a guarantee of correctness.
Conversely, some important and correct papers were never published in
refereed journals~-- one example is the work of Perelman on
the Poincar\'e conjecture, which solved one of the Millennium Prize Problems,
but was only published on the preprint server arXiv.}

At the present time, all that we can say with certainty is that the status of 
Mochizuki's proof is unclear. For further information, see
\cite{Klarreich,Scholze-Stix-2018,Woit1},
and comments on MathOverflow.

\subsection{Summary of the examples above}	\label{subsec:summary}

\begin{quotation}
\emph{A new scientific truth does not triumph by convincing its
\hbox{opponents} and
making them see the light, but rather because its opponents eventually die
and a new generation grows up that is familiar with it.}\\
\hspace*{\fill}Max Planck~\cite{Planck}
\end{quotation}

We have given examples of errors that were found and
corrected, see \S\ref{subsec:Poincare} (Poincar\'e) and
\S\ref{subsec:Wiles} (Wiles).
Also, we have mentioned some proofs that were fatally flawed,
but a different (correct) proof of the result 
was found later, see \S\ref{subsec:four-colour} 
(Kempe, Tait, Appel and Haken).
In other cases an error was found and acknowledged by the author,
and no correct proof of the result is known,
see \S\ref{subsec:Rademacher} (Rademacher).

Yet another category is where a proof is unlikely to be correct,
but this is impossible to verify, as the author has died and the
claimed proof was never written down (or has been lost), 
see \S\ref{subsec:Stieltjes} (Stieltjes).
Fermat's claimed proof (of FLT) is in this category.

Finally, we have given one example of a proof that is disputed, in that the
author maintains that it is correct, but a significant number of experts
disagree, see \S\ref{subsec:Mochizuki}. In such cases, time will tell~--
eventually the proof will be accepted as correct,\footnote{Possibly
after being rewritten to fill in gaps and make it more readily
comprehensible.
}
or the author and his supporters will
die and the proof will be relegated to ``the dustbin of history''.

\section{Ag\'elas and Vassilev-Missana}		\label{sec:VMA}

In March 2021, L\'eo Ag\'elas sent the author a preprint
that had appeared online~\cite{Agelas} and had also
been submitted to a journal.
In~\cite{Agelas}, Ag\'elas states\footnote{We use the word ``Claim''
for a statement that we may later prove to be false, in order to
distinguish it from a statement that we believe to be true.}

\begin{claim}[Ag\'elas, Theorem 2.1]		\label{claim:1}
For any Dirichlet character $\chi$ modulo $k$, the Dirichlet L-function
$L(\chi, s)$ has all its non-trivial zeros on the critical line
$\Re(s) = \frac12$.
\end{claim}
This is the \emph{Generalized Riemann Hypothesis} (GRH), probably
first 
formulated by Adolf Piltz in 1884 
(see Davenport~\cite[p.~124]{Davenport}).	
A special case, which corresponds to the principal character
$\chi_0(n) = 1$ and the Riemann zeta-function $\zeta(s)$,
is the Riemann Hypothesis, cf.\ \S\ref{subsec:Rademacher}.

Ag\'elas defines the half-plane
${\mathcal A} := \{s\in\C : \Re(s) > 1\}$, and
two Dirichlet series (convergent for $s \in \mathcal A$):
\[ 
P(\chi,s) := \sum_{p\in\primes}\chi(p)p^{-s}
\]
and
\[ 
P_2(\chi,s) := \sum_{p\in\primes}\chi(p)^2p^{-s},
\]
where $\primes$ is the set of primes $\{2, 3, 5, \ldots\}$.

When trying to understand the proof of Claim~\ref{claim:1} 
by Ag\'elas, we considered the case of the Riemann zeta-function.
Since this was sufficient to find an error in the proof, we
only need to consider this case. Thus we can take $\chi(p) = 1$, 
so $P(\chi,s)$ and $P_2(\chi,s)$ both reduce to
the usual \emph{prime zeta function}~\cite{Froberg}
\[
P(s) := \sum_{p\in\primes}p^{-s}.
\]
It is well-known~\cite[p.~188]{Froberg} that, 
for $\Re(s) > 1$,
\begin{equation}				\label{eq:P_Mobius}
P(s) = \sum_{k=1}^\infty \frac{\mu(k)}{k}\log\zeta(ks).
\end{equation}
Vassilev-Missana~\cite{Vassilev-Missana} states
\begin{claim}[Vassilev-Missana, Theorem 1]		\label{claim:2}
For integer%
\footnote{It is not clear why Vassilev-Missana imposes such a
strong restriction on $s$; we might expect the relation to hold
for all $s \in \mathcal A$, or even (using analytic continuation)
for almost all $s \in \{z\in\C: \Re(z) > 0\}$.}
$s > 1$, the relation
\[
(1-P(s))^2 = \frac{2}{\zeta(s)} - 1 + P(2s)
\;\;\text{ holds.}
\]
\end{claim}

Ag\'elas states 
\begin{quotation}
\emph{Lemma 2.3 appears as an extension of Theorem 1 of Vassilev-Missana (2016),
we give here the details of the proof as it is at the heart of the Theorem
obtained in this paper. For this, we borrow the arguments used in 
Vassilev-Missana (2016).}
\end{quotation}
\pagebreak[3]

\noindent He then states
\begin{claim}[Ag\'elas, Lemma 2.3]			\label{claim:3}
For $s \in \mathcal A$, we have
\[
(1-P(\chi,s))^2L(\chi,s) - (P_2(\chi,2s)-1)L(\chi,s) = 2.
\]
\end{claim}
In the case that we consider, namely $L(\chi,s) = \zeta(s)$,
both Claim~\ref{claim:2} and Claim~\ref{claim:3} amount to the same
relation, which we can write in an equivalent form as
\begin{equation}					\label{eq:the_claim}
\frac{2}{\zeta(s)} = 2 - 2P(s) + (P(s))^2 - P(2s).
\end{equation}

In \S\ref{subsec:disproof} we
show that~\eqref{eq:the_claim} is false.
This implies that Lemma~2.3 of Ag\'elas is false, as is Theorem~1 of
Vassilev-Missana.  Theorem~2.1 of Ag\'elas (the GRH) may be true,
but has not been proved. 
Theorem~2 of Vassilev-Missana is false, 
as shown in~\S\ref{subsec:Thm2}.

\subsection{Disproving claim (\ref{eq:the_claim})}	\label{subsec:disproof}

Five methods to disprove~\eqref{eq:the_claim} are given in~\cite{rpb277}.
We give three of them here. The methods that are omitted here involve analytic
continuation into the strip $0 < \Re s \le 1$.\\

\noindent{\bf Method 1.}
Expand each side of~\eqref{eq:the_claim} as a Dirichlet series
$\sum a_n n^{-s}$.
On the right-hand side (RHS), the only terms with nonzero coefficients $a_n$
are for integers $n$ of the form $p^\alpha q^\beta$, where $p$ and $q$ are
primes, $\alpha \ge 0$, and $\beta \ge 0$.  However, on the left-hand side 
(LHS), we find $a_{30} = -2 \ne 0$, since $30 = 2\times 3\times 5$ has
three distinct prime factors, implying that \hbox{$\mu(30) = -1$}.
By the uniqueness of Dirichlet series that converge absolutely for all
sufficiently large values of $\Re(s)$ \cite[Thm.~4.8]{Hildebrand},
we have a contradiction, so~\eqref{eq:the_claim} is false. 
\begin{remark}
{\rm
Instead of $30$ we could take any squarefree positive integer 
with greater than two prime factors.  
} 
\end{remark}

\noindent{\bf Method 2.} 
We can evaluate both sides of~\eqref{eq:the_claim} numerically
for one or more convenient values of~$s$.
If we take $s = 2k$ for some positive integer $k$, then
the LHS of~\eqref{eq:the_claim}
can easily be evaluated using Euler's formula
\[\zeta(2k) = \frac{(-1)^{k-1}(2\pi)^{2k}}{2\cdot (2k)!}\,B_{2k}\,,\]
where $B_{2k}$ is a Bernoulli number.
The RHS can be evaluated by using~\eqref{eq:P_Mobius}.
Taking $k=1$, i.e.\ $s=2$, the LHS is
$12/\pi^2 = 1.2158542$ and the RHS is $1.2230397$ 
(both values correct to $7$ decimals). 
Thus, $|\text{LHS}-\text{RHS}| > 0.007$. 
This is a contradiction, so~\eqref{eq:the_claim} is false.	
\begin{remark}					\label{remark:numer_verif}
{\rm
It is always a good idea to verify identities numerically whenever it is
convenient to do so. A surprising number of typographical 
and more serious errors
can be found in this manner. Early mathematicians such as Euler, Gauss,
and Riemann were well aware of the value of numerical computation, even
though they lacked the electronic tools and mathematical software
(such as Maple, Magma, Mathematica, SAGE, \ldots) that are available today.

If we had followed the philosophy of 
experimental mathematics~\cite{Borwein}, we would have attempted
method~2 first. However, method~1 has the advantage that all computations
are easy to do by hand (or mental arithmetic). Method~2 is slightly more
work, as it requires writing a small program.
} 
\end{remark}

\noindent{\bf Method 3.}
We consider the behaviour of each side of~\eqref{eq:the_claim} near
$s = 1$. On the LHS there is a simple zero at $s=1$, since the denominator
$\zeta(s)$ has a simple pole. On the RHS there is a logarithmic
singularity of the form
$a(\log(s-1))^2 + b\log(s-1) + O(1)$. 
This is a contradiction, so~\eqref{eq:the_claim} is false.	

\pagebreak[3]

\subsection{Theorem 2 of Vassilev-Missana is false}	\label{subsec:Thm2}

Vassilev-Missana~\cite[Theorem 2]{Vassilev-Missana} makes the following
claim.
\begin{claim}					\label{claim:4}
For integer
 $s > 1$,
\begin{equation}				\label{eq:thm2}
P(s) =
1-\sqrt{2/\zeta(s) - \sqrt{2/\zeta(2s) - \sqrt{2/\zeta(4s) - 
  \sqrt{2/\zeta(8s) - \cdots}}}}
\end{equation}
\end{claim}

\begin{proof}[Proof that Claim $\ref{claim:4}$ is incorrect]
Assume that Claim~$\ref{claim:4}$ is correct.
Replacing $s$ by $2s$ and using the result
to simplify~\eqref{eq:thm2}, we obtain
\begin{equation}			\label{eq:reverse-eng}
1 - P(s) = \sqrt{2/\zeta(s) - (1-P(2s))}.
\end{equation}
Squaring both sides of~\eqref{eq:reverse-eng} and simplifying 
gives~\eqref{eq:the_claim}, but we showed in \S\ref{subsec:disproof}
that~\eqref{eq:the_claim} is incorrect. This contradiction shows
that~Claim~$\ref{claim:4}$ is incorrect.
\end{proof}

\begin{remark}
{\rm
An alternative is to resort to a variation on method~2 above.
With $s=2$
we find numerically that 	
$P(s) \approx 0.4522$		
and 
\[
1-\sqrt{2/\zeta(s) - \sqrt{2/\zeta(2s) - \sqrt{2/\zeta(4s) - \cdots}}}
\approx 0.4588 \ne P(s),	
\]
where the numerical values are correct to $4$ 
decimal places. Thus,~\eqref{eq:thm2} is incorrect.
} 
\end{remark}
\begin{remark}					\label{remark:defn}
{\rm
It may not be clear what the infinite expression
on the RHS of~\eqref{eq:thm2} means. 
We state Claim~\ref{claim:4} more precisely
as
\begin{equation}				\label{eq:precise}			
P(s) = 
1- \lim_{n\to\infty}
   \sqrt{2/\zeta(s) - \sqrt{2/\zeta(2s) - \sqrt{2/\zeta(4s) - \cdots 
  \sqrt{2/\zeta(2^n s)}}}}\;.
\end{equation}
The limit exists and is real if $s$ is real, positive, and sufficiently
large.

To evaluate~\eqref{eq:precise} numerically, we could start with
a sufficiently large value of $n$, then evaluate the nested
square roots 
in~\eqref{eq:precise} by working from right to left, using the
values of $\zeta(2^n s), \zeta(2^{n-1}s), \ldots, \zeta(2s), \zeta(s)$.
In fact, it is desirable to replace $2/\zeta(2^ns)$ in~\eqref{eq:precise}
by $2/\zeta(2^ns)-1$, as this gives the same limit but with faster
convergence (for details see~\cite{rpb277}).
} 
\end{remark}

\section{Errors in, or relevant to, the author's own work} \label{sec:Brent}

In this section we discuss three of the author's papers. The first,
concerned with the analysis of the binary Euclidean algorithm, had some
significant errors
which were only noticed (and corrected) $21$ years after the paper was
published, although the paper had been referred to several times in the
intervening period.  

The second paper (\S\ref{subsec:rpb055}), 
concerning integer multiplication,
contained a conjecture which, although believed by the
authors and supported by numerical evidence, was false.
This was shown by Erd\H{o}s in a 1960 paper~\cite{Erdos1960}
that, unfortunately, was written
in Russian and difficult to access. Moreover, it
turned out that Erd\H{o}s's proof was incorrect, and was only corrected
by Erd\H{os} and S\'ark\"ozy some $20$ years later~\cite{Erdos1980}.

The third paper (\S\ref{subsec:rpb269}) pointed out
an (incorrect) implicit assumption made in several papers, by various authors, 
concerning fast algorithms based on the arithmetic-geometric mean.

\subsection{Analysis of the binary Euclidean algorithm}	\label{subsec:rpb037}

\begin{quotation}	
\emph{Tam complicat\ae\ evadunt, ut nulla spes superesse 
videatur.}\footnote{\emph{They come out so complicated that no hope appears 
to be left}  (Gauss referring to his analysis of the Euclidean algorithm).}\\
\hspace*{\fill}Gauss, notebook, 1800
\end{quotation}

My 1976 paper 
\cite{rpb037} proposed a heuristic probabilistic model for the 
\hbox{binary} Euclidean algorithm.  
Some forty years later, the heuristic assumptions of the  model
were fully justified by Ian Morris~\cite{Morris},
building on earlier work by Brigitte Vall\'ee~\cite{Vallee}
and G\'erard Maze~\cite{Maze}.
The paper~\cite{rpb037} 
contained some significant errors which were not
noticed until 1997, when Donald Knuth was revising volume~$2$ of his
classic series \emph{The Art of Computer Programming} in preparation for
publication of the third edition~\cite{Knuth}. The errors
all take the same form, which we illustrate by considering a typical
case. Further details are given in~\cite[\S9]{rpb183}.

It is convenient to define $\lg x := \log_2 x = (\ln x)/\ln 2$.
Consider the function
$f:[0,\infty) \mapsto (0,\infty)$ defined by
\begin{equation}
f(x) := \sum_{k=1}^\infty \frac{2^{-k}}{1+2^k x}\,.	\label{eq:fD1}
\end{equation}
Observe that
\[
f'(x) = - \sum_{k=1}^\infty \left(\frac{1}{1+2^k x}\right)^2,
\]
where the series converges for all $x > 0$, and
as $x \to 0+$ we have $f'(x) \sim \lg x$.
Thus, we might expect $f(x)$ to have a logarithmic singularity of the form
$x\lg x$ at the origin (though this turns out to be incorrect, see below).

In~\cite[Lemma~3.1]{rpb037}, it is claimed that 
(with $f$ denoted by $D_1$ in~\cite{rpb037}),
\begin{equation}
f(x) = 1 + x\lg x + \frac{x}{2}
	-\frac{x^2}{1+x} + \sum_{k=1}^\infty\frac{(-1)^{k}x^{k+1}}{2^k-1}
						\label{eq:L3.1-incorrect}
\end{equation}
for $0 < x < 2$, but this is incorrect, as we shall show.

It was pointed out by Flajolet and Vall\'ee [personal communication, 1997]
that we can obtain an equivalent 
expression for $f(x)$ using Mellin transforms~\cite[App.~B.7]{FS}.
The Mellin transform of $g(x) := 1/(1+x)$ is
\[
g^*(s) = \int_0^\infty g(x)x^{s-1}\,dx = \frac{\pi}{\sin \pi s}
\]
in the strip $0 < \Re s < 1$.
Now $f(x) = \sum_{k\ge 1} 2^{-k}g(2^k x)$, so the Mellin transform of $f(x)$
is
\[ f^{*}(s) = \sum_{k=1}^\infty 2^{-k(s+1)}g^{*}(s)
        = \frac{g^{*}(s)}{2^{s+1} - 1}\;. \]
Using the inverse Mellin transform,
we can write $f(x)$ as a sum of residues~of
\begin{equation*}
h(s) := \left(\frac{\pi}{\sin \pi s}\right)\frac{x^{-s}}{2^{s+1} - 1} 
\end{equation*}
for $\Re s \le 0$. 
The function $h(s)$ has poles for $s \in \Z$ 
(where $\sin \pi s = 0$),
and also for $s = -1 + 2\pi i n/\ln 2$, $n\in\Z$
(where $2^{s+1}=1$). Note that there is a double pole at $s = -1$.
Evaluating the residues gives, for $x\in (0,1)$,
\begin{equation}					\label{eq:correctf2}
f(x) = 1 + x\lg x + \frac{x}{2}
	-\frac{x^2}{1+x} + \sum_{k=1}^\infty\frac{(-1)^{k}x^{k+1}}{2^k-1}
	+ xP(\lg x),
\end{equation}
where
\begin{equation}
P(t) = \frac{2\pi}{\ln 2}\sum_{n=1}^\infty \frac{\sin 2n\pi t}{
        \sinh(2n\pi^2/\ln 2)}		                \label{eq:Pt}
\end{equation}
is a small periodic function arising from the non-real poles of $h(s)$.

We observe that the correct expression~\eqref{eq:correctf2} differs from the
incorrect~\eqref{eq:L3.1-incorrect} 
precisely by the addition of the small term $xP(\lg x)$.

The reason for the error in the proof of Lemma~3.1 of \cite{rpb037} is that
it was simply \emph{assumed} that $f(x)$ could be written as
$\gamma(x)\lg(x) + \delta(x)$, where $\gamma(x)$ and $\delta(x)$ are
analytic and regular in the unit disk $|x| < 1$. 
Expressions for $\gamma(x)$ and $\delta(x)$ were then deduced, 
giving~\eqref{eq:L3.1-incorrect}. The error was not in the deduction
of~\eqref{eq:L3.1-incorrect}, but in the incorrect assumption regarding
the form of $f(x)$.

In retrospect, it should have been obvious that, if $x$ is regarded
as a complex variable, then~\eqref{eq:fD1} defines an analytic function
with poles at \hbox{$x = -2^{-k}$} for $k = 1, 2, 3, \ldots$. On the other hand,
$\gamma(x)\lg(x) + \delta(x)$ has only one singularity in the disk
$|x| < 1$, and that is the logarithmic singularity at $x=0$. Thus, as in
the example of \S\ref{subsec:disproof}, the singularities differ.

The reader may ask why we did not follow our own advice 
(see Remark~\ref{remark:numer_verif} in \S\ref{subsec:disproof}) 
and attempt to verify~\eqref{eq:L3.1-incorrect}
numerically.
In fact, we did verify the equality using floating point
arithmetic on the computer available to us at the time (1976).\footnote{We
used a Univac 1100/42 mainframe with a $36$-bit wordlength, and $27$ bits
for the floating-point fraction, equivalent to about $8$ decimals.
This was in the
days before personal computers or the IEEE 754 standard for
floating-point arithmetic, or the widespread availability of software for
high-precision arithmetic.}
This was insufficient to show a discrepancy, because
$|P(t)| < 7.8\times 10^{-12}$ for $t\in\R$.  
The reason why $P(t)$ is so small is that the denominators
$\sinh(2n\pi^2/\ln 2)$ in~\eqref{eq:Pt} are large;
the smallest denominator is $\sinh(2\pi^2/\ln 2) > 1.16\times 10^{12}$.
Nowadays, we would attempt a verification to much higher precision
(say $40$ decimals), 
and this would be sufficient to detect the discrepancy caused by
the term $xP(\lg x)$ in \eqref{eq:correctf2}.

The analysis of the binary Euclidean algorithm predicts that the 
\hbox{expected}
number of iterations is $\sim K\lg n$ for uniformly distributed
$n$-bit inputs. Here $K$ is a constant that can be expressed as an
integral, where the integrand includes a term involving $P(t)$.
In 1997, Knuth attempted to evaluate $K$ accurately, 
and obtained\\[-10pt]
\[K = 0.70597\,12461\,019\underbar{45}\cdots\,,\]\\[-22pt]
but I had found
\[K = 0.70597\,12461\,019\underbar{16}\,\cdots\,.\]\\[-15pt]
In a curious twist,
it turned out that Knuth's value was incorrect, because he relied on
some of the incorrect results in my paper~\cite{rpb037},
whereas my value was correct, because I had used a more direct numerical
method that depended only on recurrences for certain
distribution functions that were given correctly in~\cite{rpb037}.
With assistance from Flajolet and Vall\'ee, we reached agreement on
the correct value of $K$ just in time to meet the deadline for
the third edition of~\cite{Knuth}.

\subsection{The Brent-Kung multiplication paper}	\label{subsec:rpb055}

\begin{quotation}
\emph{If you only want him to be able to cope with addition and subtraction,
then any French or German university will do.  But if you
are intent on your son going on to multiplication and division~-- assuming
that he has sufficient gifts~-- then you will have to send him to
Italy.}\\
\hspace*{\fill} $15$-th century advice, 
quoted by Georges Ifrah~\cite[pg.~577]{Ifrah}.
\end{quotation}

In a 1982 paper~\cite{rpb055} with H.~T.~Kung,
we considered the area $A$ and/or time $T$ required 
to perform multiplication of $n$-bit
integers expressed in binary notation.
We showed that, in a certain realistic model of computation,
there is an area-time tradeoff, and $AT = \Omega(n^{3/2})$.
To obtain this bound we needed a lower bound on the function $M(N)$ defined
by\footnote{In~\cite{rpb055}, we used the notation $\mu(N)$ instead of 
$M(N)$.}
\begin{equation}				\label{eq:MN_def}
M(N) := |\{ij\,|\,0 \le i < N,\, 0 \le j < N\}|.
\end{equation}
\begin{remark}
{\rm In the number theory literature, it is customary to define
\[M(N) := |\{ij\,|\,0 < i \le N,\, 0 < j \le N\}|,\] 
which may be interpreted as the number of distinct entries in
an $N \times N$ multiplication table.
However, when considering $n$-bit binary multiplication, the
definition~\eqref{eq:MN_def} (with $N = 2^n$) is more natural.
Either definition can be used; it makes no difference to the
asymptotics.}
\end{remark}
For our purposes, it was sufficient to use the easy lower bound%
\footnote{Here and elsewhere, ``$\log$'' denotes the natural logarithm and
``$\lg$'' denotes the logarithm to base~$2$. 
Note that, in~\cite{rpb055},
``$\log$'' denotes the logarithm to base~$2$
and ``$\ln$'' is used for the natural logarithm.}
\[M(N) \ge \frac{N^2}{2\log N}\;\;\text{for all } N \ge 4.\]
We also investigated $M(N)$ numerically and, as a result of the
numerical evidence (see \cite[Table~II]{rpb055}),
\emph{conjectured} that
\begin{equation}				\label{eq:BK_conjecture}
\lim_{N\to\infty} \left(\frac{M(N)\,\lg\lg N}{N^2}\right) = 1.
\end{equation}

As numerical evidence for this conjecture, we found that,
for $5 \le n \le 17$ and $N = 2^n$, the following inequality holds:
\[
0.995 < M(N)/M^*(N) < 1.007, \text{ where } M^*(N) =
\frac{N^2}{0.71+\lg\lg N}\,.
\]
The constant $0.71$ here was chosen empirically to give a good fit to the
data; it does not affect the conjecture since
$M^*(N) \sim N^2/\lg\lg N$ as $N \to \infty$.

Shortly after publication of \cite{rpb055},
Paul Erd\H{o}s, in a letter to the author,  pointed out that 
the conjecture~\eqref{eq:BK_conjecture}
is \emph{false}, since it contradicts a result in a paper~\cite{Erdos1960}
that he published in 1960 (albeit in a rather inaccessible Russian journal).
In fact, he showed that
\begin{equation}			\label{eq:Erdos1960}
M(N) = \frac{N^2}{(\log N)^{c + o(1)}}\,,
\end{equation}
where
\[
c = 1 - \frac{1+\ln\ln 2}{\ln 2} \approx 0.086
\]
is a small positive constant.
Much later (in 2008), Ford~\cite[Corollary~3]{Ford} gave the more precise result
\begin{equation}			\label{eq:Ford}
M(N) \asymp \frac{N^2}{(\log N)^c (\log\log N)^{3/2}}\,.
\end{equation}
Erd\H{o}s's result~\eqref{eq:Erdos1960}
contradicts~\eqref{eq:BK_conjecture},
since $\log N$ grows faster than any power of $\log\log N$
as $N \to \infty$.
However, as the functions $(\log N)^c$ and $\log\log N$ grow very slowly, 
the true asymptotic rate of growth
of $M(N)$ was not evident from computations
with $N \le 2^{17}$.
For later computations with larger values of $N$, see~\cite{rpb272}.

We did, in some sense, have the last laugh, as it turned out that Erd\H{o}s's
proof of~\eqref{eq:Erdos1960} in~\cite{Erdos1960} was incorrect,
as it used a known result outside its domain of applicability.
This was pointed out by Karl Norton,	
and the proof was corrected in 
Erd\H{o}s and S\'ark\"ozy~\cite{Erdos1980}.\footnote{Erd\H{o}s 
and S\'ark\"ozy~\cite{Erdos1980} was published in 1980, 
so it is surprising that
Erd\H{o}s did not mention it to me in his letter of 1981.
I only became aware of it forty years later,
by following a chain of references on MathSciNet.
}

\subsection{Equivalence of some AGM algorithms}		\label{subsec:rpb269}

We now describe an \emph{implicit} error, where several authors made an
implicit assumption that was later shown to be incorrect.
Fortunately, this had no serious consequences.

The asymptotically fastest known algorithms for the high-precision
computation of $\pi$ are based on the arithmetic-geometric mean (AGM)
of Gauss and Legendre.
The first such algorithm was discovered in 1976 by the author~\cite{rpb034}
and (independently) by Salamin~\cite{Salamin76}. It is sometimes called the
\emph{Gauss-Legendre} algorithm, since it is based on results that can be
found in the work of these two mathematicians~\cite{rpb252}.
Subsequently, several other AGM-based algorithms were found by the
Borwein brothers~\cite{PAGM}.

We give two of the AGM-based algorithms, 
the Gauss-Legendre \hbox{algorithm} (GL1),
and the Borwein-Borwein quartic algorithm (BB4)~\cite{PAGM},
as presented in~\cite[\S4]{rpb269}.\\

\begin{samepage}
\noindent\textbf{Algorithm GL1}\\
\textbf{Input}: The number of iterations $n_{max}$.\\
\textbf{Output}: A sequence $(\pi_n')$ of approximations to $\pi$.
\begin{align*}
&a_0 := 1;\; b_0 := 1/\sqrt{2};\; s_0 := \textstyle\frac{1}{4};\\
&\textbf{for } n \text{ from } 0 \text{ to } n_{max}-1 \text{ do }\\
&\hspace*{2em}a_{n+1} := (a_n+b_n)/2;\\
&\hspace*{2em}c_{n+1} := a_n - a_{n+1};\\       
&\hspace*{2em}\textbf{output } \pi_n' := a_{n+1}^2/s_n.\\
&\hspace*{2em}\textbf{if } n < n_{max}-1 \textbf{ then}\\[-5pt]
&\hspace*{4em}b_{n+1} := \sqrt{a_n b_n};\\
&\hspace*{4em}s_{n+1} := s_n - 2^n\,c_{n+1}^2.\\[-30pt]
\end{align*}
\end{samepage}

\begin{remark}
{\rm
Subscripts on variables such as $a_n, b_n$ are given for
expository purposes. In an efficient implementation only a constant number of
real variables are needed, because $a_{n+1}$ can overwrite $a_n$
(after saving $a_n$ in a temporary variable for use in the computation
of $c_{n+1}$), and similarly for $b_n$, $c_n$, $s_n$, and $\pi_n'$.
Similar comments apply to Algorithm BB4.
} 
\end{remark}

\begin{samepage}
\noindent\textbf{Algorithm BB4}\\
\textbf{Input}: The number of iterations $n_{max}$.\\
\textbf{Output}: A sequence $(\pi_n'')$ of approximations to $\pi$.
\begin{align*}
&y_0 := \sqrt{2}-1;\;\; z_0 := 2y_0^2;\\
&\textbf{for } n \text{ from } 0 \text{ to } n_{max}-1 \text{ do}\\
&\hspace*{2em}\textbf{output } \pi_n'' := 1/z_n\,;\\
&\hspace*{2em}\textbf{if } n < n_{max}-1 \textbf{ then }\\
&\hspace{4em}y_{n+1} := \frac{1-(1-y_n^4)^{1/4}}
                  {1+(1-y_n^4)^{1/4}}\;;\\
&\hspace{4em}z_{n+1} := 
        z_n(1+y_{n+1})^4 - 2^{2n+3}y_{n+1}(1+y_{n+1}+y_{n+1}^2).\\[-20pt]
\end{align*}
\end{samepage}

Algorithm GL1 produces a sequence $(\pi_n')$ of approximations to $\pi$.
It is known that Algorithm GL1 has \emph{quadratic} convergence,
so (roughly speaking) the number of correct digits \emph{doubles}
at each iteration.
More precisely, if $e_n' := \pi - \pi_n'$ is the error in $\pi_n'$, then 
\begin{equation}				\label{eq:error-GL1}
0 < e_{n}' < \pi^2 2^{n+4}\exp(-2^{n+1}\pi).
\end{equation}

Similarly, Algorithm BB4 produces a sequence $(\pi_n'')$ of approximations
to~$\pi$, and the sequence has \emph{quartic} convergence, 
so (roughly speaking) the number of correct digits \emph{quadruples}
at each iteration.
More precisely, if $e_n'' := \pi - \pi_n''$ is the error in $\pi_n''$, then 
\begin{equation}				\label{eq:error-BB4}
0 < e_{n}'' < \pi^2 2^{2n+4}\exp(-2^{2n+1}\pi).
\end{equation}

In 2017--2018, we observed that the error bound on $e_{2n}'$, obtained
from~\eqref{eq:error-GL1} by the substitution $n \to 2n$,
is \emph{exactly} the same as the
error bound on $e_n''$ given in~\eqref{eq:error-BB4}.
We computed some of the errors to high precision, and discovered that, 
not only were the error bounds equal, but the errors were identical.
More precisely, $e_n'' = e_{2n}'$ for all $n \ge 0$.

This implies that one iteration of Algorithm BB4 is equivalent to
two iterations of Algorithm GL1, in the sense that
$\pi_n'' = \pi_{2n}'$.

In the literature up to 2018, it had been (implicitly) \emph{assumed}
by the Borwein brothers, Bailey, Kanada, and others
~\cite{Bailey88,BBB89,PAGM,Kanada} 
that the algorithms were inequivalent.\footnote{The equivalence result assumes
that arithmetic is exact. If, as is necessary in practice, limited-precision
approximations are used, then $\pi_n''$ and $\pi_{2n}'$ may differ
slightly, due to the effect of rounding errors.}
For example, when computing $\pi$ to high precision,
Kanada~\cite{Kanada} used both algorithms as a consistency
check. Although this would catch some programming errors, it does
not provide an independent check that the constant computed is
actually~$\pi$. A better consistency check would be provided by a fast algorithm
based on different theory,
such as a Ramanujan-Sato series for $1/\pi$, 
see~\cite{BBB89}, \cite[\S6]{rpb269}.
The equivalence is implicit in the work of Guillera~\cite{Guillera08},
but was not stated explicitly until 2018, when a proof was given
in~\cite{rpb269}. A different proof may be found in~\cite{Milla}.

\begin{remark}
{\rm
An iteration of Algorithm BB4 is about twice as time-consuming as an
iteration of Algorithm GL1.
On the other hand, Algorithm GL1 requires twice as many iterations
to obtain the same precision. Thus, it is not clear which algorithm
is preferable in practice. For more on this topic, 
see~\cite[end of \S4]{rpb269}.
}
\end{remark}

\section{Some conclusions}				\label{sec:conclusion}

There are many classes of mathematical errors. We list some of the
more common ones. An awareness of such errors may help the
reader to avoid similar ones.

\begin{enumerate}

\item Numerical or algebraic errors,
possibly caused by an incorrect or numerically
unstable algorithm, or an error in its implementation as a computer
program.

\item Errors of omission.
For example, Poincar\'e missed the possibility of chaotic 
behaviour (\S\ref{subsec:Poincare}),
and Vassilev-Missana failed to consider integers with 
more than two distinct prime factors (\S\ref{subsec:disproof}).

\item Unwarranted assumption, such as in our analysis of the binary
Euc\-lidean algorithm (\S\ref{subsec:rpb037}).
Many incorrect proofs of RH are in this class.
For example, they may assume that some property of $\zeta(s)$ that holds
for $\Re(s) > 1$
also holds for $1/2 \le \Re(s) \le 1$.

\item Gap in the proof. A recent example is Wiles's first proof of FLT,
see \S\ref{subsec:Wiles}. More generally, perhaps
the author proves A and claims that A implies B, but the implication is not
obvious~-- it may be true but needs to be proved. Also, a proof may be
logically correct, so far as it goes, but not prove what is claimed.
For example, if we are trying to prove a statement $(\forall n\ge 0) P(n)$,
it is not (usually) sufficient to prove $P(0), P(1), \ldots, P(99)$.

\item Using an incorrect result from a published paper.
For example, Ag\'elas (\S\ref{sec:VMA}) used an incorrect
result of Vassilev-Missana.
In the worst case this could lead to a whole tree of
incorrect results. 

\item Using a correct result but applying it incorrectly.
For example, the definitions may be subtly different, 
or the domain of applicability of the correct
result may not be taken into account correctly. See for example
the discussion of
Erd\H{o}s's papers~\cite{Erdos1960,Erdos1980} in \S\ref{subsec:rpb055}.

\item General lack of clarity or rigor, 
so although the proof may be correct in some
sense, it is not currently accepted by the mathematical community
(see \S\ref{subsec:Mochizuki}).
Old examples include many of Euler's proofs,
and various proofs of the fundamental theorem of
algebra that were not rigorous by modern standards.
More recently, some well-known examples are from the Italian school of 
algebraic geometry~\cite{Wiki-ItalianSchool} in the period 1885--1935. 

\end{enumerate}

\subsection*{Acknowledgements}

\S\ref{sec:VMA} is a summary of material in~\cite{rpb277,rpb278}.
Kannan Soundararajan pointed out some relevant discussion
on MathOverflow~\cite{Mathoverflow-Klangen}.
We thank L\'eo Ag\'elas for his correspondence
regarding \S\ref{sec:VMA},
and L\'eo Ag\'elas, Rob Corless, and Artur Kawalec for confirming some
of the computations in~\S\ref{sec:VMA}.

We thank Philippe Flajolet\footnote{Philippe Flajolet 1948--2011.},
Don Knuth, and Brigitte Vall\'ee,
for their correspondence regarding~\S\ref{subsec:rpb037},
and 
Paul Erd\H{o}s\footnote{Paul Erd\H{o}s 1913--1996.}, 
Donald J.~Newman\footnote{Donald J.~Newman 1930--2007.}, 
Andrew Odlyzko, and
Carl Pomerance, for correspondence 
regarding the conjecture of~\S\ref{subsec:rpb055}.
We also thank Jonathan Borwein\footnote{Jonathan Borwein 1951--2016.}
and David Bailey for discussions regarding AGM-based algorithms,
relevant to \S\ref{subsec:rpb269}.

Finally, thanks to 
Nathan Clisby,
Rob Corless,
Garry Herrington,
Fredrik Johansson,
Christian Krattenthaler,
Cleve Moler,
Nick Trefethen,
and Tim Trudgian
for their comments on the
first draft of this paper, and for noticing some typos.

\pagebreak[3]

\end{document}